\documentclass[11pt,thmsa]{article}%
\usepackage{amsmath}
\usepackage{amsfonts}
\usepackage{amssymb}
\usepackage{graphicx}%
\newtheorem{theorem}{Theorem}[section]

\newtheorem{corollary}[theorem]{Corollary}

\newtheorem{lemma}[theorem]{Lemma}

\newenvironment{proof}[1][Proof]{\noindent\textbf{#1.} }{\ \rule{0.5em}{0.5em}}

\begin{document}
\title{Homogeneous Finsler sphere with constant flag curvature}
\author{Ming Xu\\
\\
School of Mathematical Sciences\\
Capital Normal University\\
Beijing 100048, P. R. China\\
Email: mgmgmgxu@163.com
\\
}
\date{}
\maketitle

\begin{abstract}
We prove that a homogeneous Finsler sphere with constant flag curvature $K\equiv1$ and a prime closed geodesic of length $2\pi$ must be Riemannian. This observation provides the evidence for the non-existence of homogeneous Bryant spheres. It also helps us propose an alternative approach
proving that a geodesic orbit Finsler sphere with $K\equiv1$ must be Randers. Then we discuss the behavior of geodesics on a homogeneous Finsler sphere with $K\equiv1$. We prove that many geodesic properties for homogeneous Randers spheres with $K\equiv1$ can be generalized to the non-Randers case.


\end{abstract}

\section{Introduction}
The classification of Finsler spheres $(S^n,F)$ with $n>1$ and constant flag curvature $K\equiv1$ is one of the most intriguing open problems in
Finsler geometry. The Riemannian one is unique, i.e. it must be the unit
sphere ${S}^n(1)\subset\mathbb{R}^{n+1}$ with the submanifold metric induced from the Euclidean space, which we will simply call the standard Riemannian metric. The Randers ones are classified by
D. Bao, C. Robles and Z. Shen, which are defined by a navigation process from a standard Riemannian metric and a Killing vector field \cite{BRS}.
The affects of Killing navigation on the geometry of Finsler manifolds are well understood \cite{PM2018,HM2007,HM2011}. Though
the behavior of geodesics is much different with that of the Riemannian one \cite{An1977,Ka1973},
the Randers spheres with $K\equiv1$ can still be viewed as the
standard space forms in Finsler geometry.

There exist much more complicated Finsler metrics on spheres with $K\equiv1$. For example, R. Bryant
constructed non-Randers Finsler metrics on $S^n(1)$ with $K\equiv1$, such that their geodesics (as point sets) are great circles \cite{Br1996,Br1997,Br2002}. We will simply call them
{\it Bryant spheres} for simplicity.
A significant feature of Bryant spheres is that they are  projectively flat. All geodesics of a
Bryant sphere are closed. Applying some more discussion for its antipodal map (see \cite{Sh1996} or Section 2 for this notion), it is easy to see that all prime closed geodesics on a Bryant sphere have the same length $2\pi$.

From the view point of Lie theory, we see that
a Bryant sphere $(S^n(1),F)$
may admit some degree of isometric symmetry, i.e. the connected isometry group $G=I_0(S^n,F)$ may have a positive dimension. But
until now, no
homogeneous Bryant spheres have been found.

The first main theorem of this paper explains this phenomenon.

\begin{theorem}\label{main-thm}
Assume that $(S^n,F)$ is a homogeneous Finsler sphere with $n>1$ and $K\equiv 1$, and there exists a prime closed geodesic of length $2\pi$ (or equivalently, the order of its antipodal map is 2).
Then $F$ must be Riemannian.
\end{theorem}

The proof of Theorem \ref{main-thm} is based on an observation for the geodesic orbit (g.o. in short) properties of the standard Riemannian metric $F_0$ on a unit sphere, with respect to different
transitive isometric group actions, and
a comparison between $F$ and $F_0$, concerning their antipodal maps $\psi$ and $\psi_0$
respectively, and the indicatrices they define in each tangent space.

Here are two immediate corollaries of Theorem \ref{main-thm}.
The notion of Bryant sphere requires the flag curvature $K\equiv1$ and the existence of prime closed geodesics of length $2\pi$, so we get
the first corollary, i.e.
\begin{corollary}\label{main-cor-1}
There does not exist any homogeneous Bryant sphere.
\end{corollary}

It is easy to see that on a reversible Finsler sphere $(S^n,F)$ with $n>1$ and $K\equiv1$, the order of its antipodal map is 2, so all geodesics are closed and all prime closed geodesics have the same length $2\pi$ (see Lemma \ref{lemma-0} in Section 2). So we have the second corollary
of Theorem \ref{main-thm}, i.e.

\begin{corollary}\label{main-cor-2}
Any reversible homogeneous Finsler sphere with $K\equiv1$
is Riemannian.
\end{corollary}

Notice that in 2006, R. Bryant proved the following theorem \cite{Br2006}.
\begin{theorem}\label{Bryant theorem}
Any reversible Finsler 2-sphere with $K\equiv1$
is Riemannian.
\end{theorem}
Later in 2009,
C. Kim and K. Min discussed the generalization
of Theorem \ref{Bryant theorem} to high dimensions.
Comparing their argument to that for Corollary \ref{main-cor-2},
we see that this problem is much simpler in the homogeneous context.

Using Theorem \ref{main-thm}, we can provide a more
self contained proof of the following theorem in \cite{Xu2018},
without using \cite{KM2009} (i.e. Theorem 6.2 in \cite{Xu2018}).

\begin{theorem}\label{main-thm-2}
Any geodesic orbit Finsler sphere $(S^n,F)$ with $K\equiv1$
is Randers.
\end{theorem}

By the classification of geodesic orbit Finsler spheres in \cite{Xu2018}, Theorem \ref{main-thm-2} can be equivalently stated as the following.

\begin{theorem}
Let $(S^n,F)$ be a homogeneous Finsler sphere such that it has constant flag curvature $K\equiv1$, and its connected isometry group $I_0(S^n,F)$ is not isomorphic to $Sp(k)$ if $n=4k-1$ for some positive integer $k$. Then
$F$ must be Randers.
\end{theorem}

Then we discuss the behavior of geodesics on a homogeneous Finsler sphere with $K\equiv1$. The Riemannian case and the (non-Riemannian) Randers case
are well understood. They provide models and motivations for our
previous works estimating the number of orbits of prime closed geodesics on Finsler spheres with $K\equiv1$ \cite{Xu2018-1,Xu2018-3} and homogeneous Finsler spaces \cite{Xu2018-2}. We show that many properties of the geodesics
on a homogeneous Randers sphere with $K\equiv1$ can be generalized to the non-Randers case. For the precise statement, see Theorem \ref{main-thm-3} in Section 4, which is a homogeneous analog of Theorem 2 in \cite{BFIMZ}.

Finally, we remark that, the existence of non-Randers homogeneous
Finsler spheres with $K\equiv1$ (which must be of type $Sp(k)/Sp(k-1)$ according to the classification of homogeneous spheres in \cite{MS1943}) is still an open problem. Recently, L. Huang and X. Mo constructed new examples of invariant Einstein Finsler metrics on the homogeneous sphere $Sp(k)/Sp(k-1)$ \cite{HM2018}.
Their method also sheds light on solving this open problem.

This paper is organized as following. In Section 2, we summarize some necessary knowledge on Finsler geometry and homogeneous geometry. In Section 3, we prove Theorem \ref{main-thm}, and sketch an alternative approach proving Theorem \ref{main-thm-2}
(Theorem 6.2 in \cite{Xu2018}). In Section 3, we discuss the behavior of geodesics on a homogeneous Finsler sphere with $K\equiv1$. In particular, we propose Theorem \ref{main-thm-3},
concerning the non-Randers homogeneous Finsler spheres with $K\equiv1$. In Section 4, we prove Theorem \ref{main-thm-3}.

\section{Preliminaries}

Firstly, we summarize some fundamental knowledge on Finsler geometry. See \cite{BCS2000,CS2005,Sh2001} for more details.

The {\it Finsler metric} on a connected smooth manifold $M^n$ is a continuous
function $F:TM\rightarrow [0,+\infty)$ satisfying the following
conditions:
\begin{description}
\item{\rm (1)} $F$ is positive and smooth when restricted to
the slit tangent bundle $TM\backslash 0$.
\item{\rm (2)} $F$ is positively homogeneous of degree one, i.e.
for any $x\in M$ and $y\in T_xM$, $F(x,\lambda y)=\lambda F(y)$ when $\lambda\geq 0$.
\item{\rm (3)} $F$ is strongly convex, i.e. for any standard local coordinates $x=(x_i)\in M$ and $y=y^j\partial_{x^j}\in T_xM$
    on $TM$, the Hessian $(g_{ij}(x,y))=(\frac12[F^2]_{y^iy^j})$
    is positive definite when $y\neq0$.
\end{description}
We will also call $(M,F)$ a {\it Finsler manifold} or a {\it Finsler space}. We say $F$ is {\it reversible} if $F(x,y)=F(x,-y)$ for any $x\in M$
and any $y\in T_xM$.

On one hand, the Hessian matrices define an inner product
$$\langle u,v\rangle_y=\frac12\frac{\partial^2}{\partial s\partial t}|_{s=t=0}F^2(y+su+tv)=g_{ij}(x,y)u^i v^j,\quad\forall u=u^i\partial_{x^i}, v=v^j\partial_{x^j}\in T_xM$$
which depends on the nonzero vector $y\in T_xM$. Sometimes, we denote this inner product as $g_y^F$ and call it the {\it fundamental tensor}.

On the other hand, the Finsler metric $F$ defines the arc length of a curve and the distance function $d_F(\cdot,\cdot)$ on $M$.
By the local minimizing principle, geodesics can be similarly defined as in Riemannian geometry.
In this paper, we will only consider geodesics $c(t)$ with positive constant speeds, i.e. $F(\dot{c}(t))\equiv\mathrm{const}>0$.

A Finsler metric is Riemannian iff the fundamental tensor $g_y^F$
is independent of the nonzero vector $y$.
The most important and simplest non-Riemannian metric is Randers metric, which is of the form $F=\alpha+\beta$, in which
$\alpha$ is a Riemannian metric and $\beta$ is a one-form. A Randers metric can be also determined by the navigation process
from the datum $({F}',V)$, in which ${F}'$ is a Riemannian metric, and $V$ is a vector field satisfying ${F}'(V)<1$ everywhere, such that $F(x, y+{F}'(y)V(x))={F}'(x,y)$ for any $x\in M$ and $y\in T_xM$. A geometrical description for the navigation process is the following. At each point $x\in M$, the indicatrix $S_x^FM=\{y\in T_xM\mbox{ with }F(x,y)=1\}\subset T_xM$ is the parallel shifting of the indicatrix $S_x^{{F}'}M$ by the vector $V(x)$.

The flag curvature $K(x,y,\mathbf{P})$ (or simply $K$ sometimes), where $\mathbf{P}=\mathrm{span}\{y,u\}$ is a tangent plane in $T_xM$, is defined as
$$K(x,y,\mathbf{P})=\frac{\langle R_yu,u\rangle_y}{
\langle u,u\rangle_y\langle y,y\rangle_y-[\langle u,y\rangle_y]^2},$$
in which $R_y:T_xM\rightarrow T_xM$ is the Riemann curvature.

The explicit presentations of geodesics and curvatures using local coordinates can be found in the references previously given.

Secondly, we introduce the antipodal map for a Finsler sphere with constant flag curvature $K\equiv1$.

Assume that $(S^n,F)$ is a Finsler sphere with $n>1$ and $K\equiv1$. Then all geodesic rays starting at $x\in M$ will meet again, after the same arc length $\pi$, at
another point $x^*\neq x$ \cite{Sh1996}. For each $x$, $x^*$ is the unique point satisfying $d_F(x,x^*)=\mathrm{diag}(S^n,F)=\pi$.

The map $\psi$ from $x$ to $x^*$ is an isometry of $(S^n,F)$ \cite{BFIMZ}. Further more, $\psi$ is a Clifford--Wolf translation which belongs to the center
of $I(S^n,F)$ \cite{Xu2018-3}. By the previous observation,
the $\psi$-orbit of $x$, i.e. $\psi^i(x)$ for all $i\in\mathbb{Z}$, is contained in any
geodesic passing $x$.

We simply call $\psi$ the {\it antipodal map} \cite{Xu2018-3}. For the standard Riemannian metric $F_0$ on the unit sphere $S^n(1)\subset\mathbb{R}^{n+1}$, its antipodal map is the classical one, i.e. $\psi_0(x)=-x$.

The order of $\psi$ in
$I(S^n,F)$ is the minimal positive integer $m$ with $\psi^m=\mathrm{id}$, or $\infty$ if such an integer does not exist.
The antipodal map for a Bryant sphere \cite{Br1996,Br1997,Br2002} has the order $m=2$.
In Finsler geometry, we will usually meet the situation that $m>2$
or $m=\infty$ \cite{BFIMZ}. The order of the antipodal map
 is a crucial index determining the behavior of geodesics on a Finsler sphere with $K\equiv1$. For example, we
have the following easy lemma.

\begin{lemma}\label{lemma-0}
Assume $(S^n,F)$ is a homogeneous Finsler sphere with $n>1$ and $K\equiv1$. Then the following statements are equivalent:
\begin{description}
\item{\rm (1)} The antipodal map $\psi$ satisfies $\psi^2=\mathrm{Id}$;
\item{\rm (2)} Each geodesic is closed and the length of each prime closed geodesic is $2\pi$;
\item{\rm (3)} Any prime closed geodesic has the same length
$2\pi$;
\item{\rm (4)} There exists a prime closed geodesic of length $2\pi$.
\end{description}
\end{lemma}

\begin{proof} We first prove the statement from (1) to (2).
Assume $\psi^2=\mathrm{id}$. Let $\gamma$ be any $F$-unit geodesic $c(t)$ with $t\in[0,\pi]$ and $c(0)=x$.
Then $\gamma$ is a shortest geodesic from $x$ to $\psi(x)=c(\pi)$. Because $\psi$ is a Clifford--Wolf translation on $(S^n,F)$, and $\psi^2=\mathrm{id}$, the closed curve $\gamma\cup \psi(\gamma)$ is smooth at $x$ and $\psi(x)$, i.e. it is
a prime closed geodesic with the length $2\pi$. This proves the statement from (1) to (2).

The statements from (2) to (3) is obvious.

The statement from (3) to (4) follows immediately the existence of closed geodesics on any closed Finsler manifold \cite{Fe1965}.
In particular, when the isometry group has a positive dimension, we can apply Lemma 3.1 in \cite{Xu2018-1} to find
two distinct prime closed geodesics.

The statement from (4) to (1) follows immediately the definition of the antipodal map and the homogeneity of $F$.

This ends the proof of this lemma.
\end{proof}

Lastly, we recall the definition of homogeneous geodesics and geodesic orbit property in Finsler geometry \cite{YD2014}.

Assume the connected Finsler manifold $(M,F)$ admits the non-trivial
isometric action of a connected Lie group $G$. We call a geodesic $c(t)$ {\it$G$-homogeneous}, if $c(t)=\exp tX\cdot x$ for some $X\in\mathfrak{g}=\mathrm{Lie}(G)$ and $x\in M$, i.e. this geodesic is the orbit of some one-parameter
subgroup of $G$. We call $(M,F)$ a {\it $G$-geodesic orbit} (or {\it g.o.} in short) Finsler space,
if all geodesics on $(M,F)$ are $G$-homogeneous. If $G$ is not specified, the assumption $G=I_0(M,F)$ is automatically taken. Obviously, any connected
$G$-g.o. Finsler space
is $G$-homogeneous.

In \cite{Xu2018}, we have classified the
geodesic orbit Finsler spheres
by the following theorem, which generalizes a theorem of Yu.G. Nikonorov in the Riemannian context \cite{Ni2013}.

\begin{theorem} \label{classification-g-o-spheres}
A homogeneous Finsler sphere $(S^n,F)$ is g.o. unless $S^n=Sp(k)/Sp(k-1)$ with $I_0(S^n,F)=Sp(k)$ for some positive integer $k$.
\end{theorem}

We will also need the following result in \cite{BN2014} or \cite{Ni2013}
 for the g.o. properties of the standard Riemannian metric on a unit sphere, with respect to different transitive isometric group actions.

\begin{lemma}\label{lemma-1}
For any closed connected subgroup $G\subset SO(n+1)$ acting transitively on the unit sphere $S^n(1)\subset\mathbb{R}^{n+1}$ with $n>1$, the standard Riemannian metric $F_0$ on $S^n(1)$ is $G$-g.o..
\end{lemma}

\section{Proofs of Theorem \ref{main-thm} and Theorem \ref{main-thm-2}}

Assume $(S^n,F)$ is a homogeneous Finsler sphere with $n>1$ and $K\equiv1$.

When $G=I_0(S^n,F)$, all possible homogeneous presentations $S^n=G/H$ are given by Table \ref{table-1}.
\begin{table}
  \centering
  \begin{tabular}{|c|c|c|c|c|}
     \hline
     No. & $n$ & $G$ & $H$ & $S^n$\\ \hline
     1   & $n>1$ & $SO(n+1)$ & $SO(n)$ & $SO(n+1)/SO(n)$\\ \hline
     2 & $n=2k-1>2$ & $U(k)$ & $U(k-1)$ & $U(k)/U(k-1)$  \\ \hline
     3 & $n=4k-1>2$ & $Sp(k)$ & $Sp(k-1)$& $Sp(k)/Sp(k-1)$\\ \hline
     4 & $n=4k-1>6$ & $Sp(k)U(1)$ &$Sp(k-1)U(1)$ & $Sp(k)U(1)/Sp(k-1)U(1)$ \\ \hline
     5 & $n=4k-1>6$ & $Sp(k)Sp(1)$ &$Sp(k-1)Sp(1)$&
     $Sp(k)Sp(1)/Sp(k-1)Sp(1)$ \\ \hline
     6 & $n=15$ & $Spin(9)$& $Spin(7)$ & $Spin(9)/Spin(7)$
     \\ \hline
   \end{tabular}
  \caption{Homogeneous spheres}\label{table-1}
\end{table}

Here are some remarks. In Case 1, the metric is Riemannian symmetric. This case covers $S^n=SO(n+1)/SO(n)$, $S^6=G_2/SU(3)$ and $S^7=Spin(7)/G_2$. In particular,
all even dimensional homogeneous spheres belong to this case. So we only
need to discuss the odd dimensional homogeneous spheres in later discussion.
For the $SU(k)$-homogeneous Finsler sphere
$(S^{2k-1},F)$ with $k>1$, it may be presented as $Sp(1)/Sp(0)$ when $k=2$, and it is $U(k)$-homogeneous when $k>2$.
In case 5, we identify $Sp(1)$ with the set of all quaternion numbers with norm one, acting on column vectors in $\mathbb{H}^k$ by right scalar multiplications. Then $G=Sp(k)Sp(1)$ represents the image of $Sp(k)\times Sp(1)$ in $SO(4k)$, such that $(A,\alpha)$ is mapped to the linear automorphism $x\mapsto Ax\alpha$ for each column vector $x\in\mathbb{H}^k$. Case 4 is similar to case 5.

Checking each case in Table \ref{table-1}, we observe that  $G=I_0(S^n,F)$ can be canonically identified as a closed subgroup of $SO(n+1)$, and meanwhile $S^n$ is identified as the unit sphere $$S^{n}(1)=\{x\in\mathbb{R}^{n+1}\mbox{ with }||x||=1\},$$
where $||\cdot||$ is the standard Euclidean norm,
such that the $G$-action on $S^n$ is induced by the left $SO(n+1)$-multiplications on column vectors.

Now on $S^n(1)$, we have two $G$-homogeneous Finsler metrics satisfying $K\equiv1$. One is the metric $F$, and the other is standard Riemannian metric $F_0$.
The following lemma indicates
their antipodal maps coincide.

\begin{lemma}\label{lemma-2}
Assume $G$ is a closed subgroup of $SO(n+1)$
acting transitively on the odd
dimensional unit sphere $S^n(1)$ with $n>2$, and $F$ is
a $G$-invariant Finsler metric on $S^n(1)$ with $K\equiv 1$
and $\psi^2=\mathrm{id}$, where $\psi$ is the antipodal map for $F$.
Then we have
$\psi(x)=-x$ for any $x\in S^n(1)$.
\end{lemma}
\begin{proof}
We may assume $G=I_0(S^n(1),F)$ and
only need to discuss the cases No. 2-6 in Table \ref{table-1}.
By Lemma \ref{lemma-0}, the assumption that $\psi^2=\mathrm{id}$ implies that all geodesics on $(S^n(1),F)$
are closed and all prime closed geodesics on $(S^n(1),F)$
have the same length $2\pi$.

We observe that for each case, the negative identity matrix $-I\in SO(n+1)$ belongs to $G$. For the cases No. 2-5, this fact is obvious.
For the case No. 6, the $Spin(9)$-action on $S^{15}(1)$ is induced by
the isotropy action for $F_4/Spin(9)$. Because $F_4/Spin(9)$
is a symmetric space with an involution $\sigma=\mathrm{Ad}(g)$
for some $g\in Spin(9)$. It implies that the isotropy action of $g$ satisfies $g\cdot x=-x$, $\forall x\in S^{15}(1)$. Another approach for the
case No. 6 is that we can identify the Euclidean space $\mathbb{R}^{16}$ as $\mathbb{O}^2$ and $Spin(9)$ as the subgroup of $SO(16)$ consisting of all elements which map Octonionic lines to Octonionic lines. Because $-I\in SO(16)$ satisfies this description, so we have $-I\in Spin(9)\subset SO(16)$. To summarize, we have proved that in each case $-I\in SO(n+1)$ is contained in $G$.

Denote $\psi_0(x)=-x$ the antipodal map for the standard Riemannian metric $F_0$ on the unit sphere $S^n(1)$.
Since $\psi_0\in C(G)$ and $G$ acts isometrically and transitively on $(S^n(1),F)$, $\psi_0$
is a Clifford Wolf translation for $F$. Take
any $x\in S^n(1)$ and any shortest geodesic $\gamma$ for $F$, from $x$ to $-x$, then
$\gamma\cup \psi_0(\gamma)$ is a prime closed geodesic, which length is $2\pi$. So the length of $\gamma$, from $x$ to $-x$, is $\pi$, for each $x\in S^n(1)$.
This ends the proof of the lemma.
\end{proof}

The coincidence between the antipodal maps suggests us to
compare $F$ and $F_0$. Then we get the following lemma.

\begin{lemma}\label{lemma-3}
Let $F$ be a $G$-invariant Finsler metric on $S^n(1)$ with $K\equiv1$ and $\psi^2=\mathrm{id}$, where
$G$ is a closed connected subgroup of $SO(n+1)$ which acts transitively on $S^n(1)$. Let $F_0$ be
the standard Riemannian metric on $S^n(1)$.
Then for any $x\in S^n(1)$ and any $v\in T_xS^n(1)$, we have $F(x,v)\geq F_0(x,v)$.
\end{lemma}

\begin{proof}
Without loss of generality, we may assume $F_0(x,v)=1$.

Let $c(t)$ be the $F_0$-unit speed geodesic on $(S^n(1),F_0)$ such that $c(0)=x$, $c(\pi)=-x$ and $\dot{c}(0)=v$.  By Lemma \ref{lemma-1}, we can find $X\in\mathfrak{g}=\mathrm{Lie}(G)$, such that $c(t)=\exp tX\cdot x$. Because the $G$-actions are isometries for $F$, each integration curve of $X$ have a constant $F$-speed, so we have $F(\dot{c}(t))\equiv F(\dot{c}(0))=F(x,v)$. By Lemma \ref{lemma-2}, $d_F(x,-x)=\pi$, so we have
$$\pi F(x,v)=\int_0^\pi F(\dot{c}(t))dt\geq d_F(x,-x)=\pi,$$
i.e. $F(x,v)\geq 1=F_0(x,v)$,
which  proves this lemma.
\end{proof}

Theorem \ref{main-thm} follows Lemma \ref{lemma-3} easily.
\bigskip

\noindent
{\bf Proof of Theorem \ref{main-thm}.}
Assume conversely that $F\neq F_0$. By Lemma \ref{lemma-3},
there exist a point $x\in M$ and a $F$-unit tangent vector $v\in T_xS^n(1)$ such that $F_0(x,v)<1$. Let $c(t)$ be the $F$-unit
speed geodesic from $x$ to $\psi(x)=-x$, with $c(0)=x$, $c(\pi)=-x$ and $\dot{c}(0)=v$. Then by Lemma \ref{lemma-3} and out assumption that $F_0(x,v)<F(x,v)=1$, the $F_0$-length of $c(t)$ from $t=0$ to $t=\pi$ satisfies
$$\int_0^\pi F_0(\dot{c}(t))dt<
\int_0^\pi F(\dot{c}(t))dt=\pi.$$
This is a contradiction because $d_{F_0}(x,-x)=d_{F_0}(x,\psi_0(x))=\pi$.

This ends the proof of Theorem \ref{main-thm}.\ \rule{0.5em}{0.5em}

In the rest of this section, we sketch an alternative proof
of Theorem \ref{main-thm-2} which does not need \cite{KM2009}.
\bigskip

\noindent
{\bf Proof of Theorem \ref{main-thm-2}.}
Without loss of generality, we assume that $F$ is a Finsler metric
on $S^n(1)$ with $n>1$ and $K\equiv1$, such that its connected
isometry group is a closed connected subgroup of $SO(n+1)$.
By Theorem \ref{classification-g-o-spheres}, we may assume that there exists a closed connected subgroup $G\subset I_0(S^n(1),F)\subset SO(n+1)$ which acts transitively on $S^n(1)$ and is presented as
in Table \ref{table-1}, except No. 3.

For the cases No. 5 and No. 6, the homogeneous spheres $S^{n}(1)=Sp(k)Sp(1)/Sp(k-1)Sp(1)$ and $Spin(9)/Spin(7)$ are
weakly symmetric, so the
$G$-invariant metric $F$ is reversible \cite{Xu2018}. In
these cases we have $\psi^2=\mathrm{id}$ because for any $x\in M$,
$d_F(x,\psi(x))=d_F(\psi(x),x)=\pi$. By Theorem \ref{main-thm},
$F$ must be Riemannian (which is also Randers).

For the cases No. 2 and No. 4, $\mathfrak{g}$ has a one-dimensional center which provides Killing vector fields of
constant length on $(S^n(1),F)$.
As shown in Section 6 of \cite{Xu2018}, after a suitable Killing navigation defined by the datum
$(F,V)$ with $V\in\mathfrak{c}(\mathfrak{g})$,
we can get a homogeneous Finsler sphere $\tilde{F}$ with $K\equiv1$ and $\psi^2=\mathrm{id}$. By Theorem \ref{main-thm},
$\tilde{F}$ is Riemannian, so $F$ must be Randers.

To summarize, in each case we have proved that $F$ is Randers,
which ends the proof of Theorem \ref{main-thm-2}.\ \rule{0.5em}{0.5em}

\section{Behavior of the geodesics on a homogeneous Finsler sphere with $K\equiv1$}

In this section, we discuss the behavior of geodesics on a homogeneous Finsler sphere $(S^n,F)$ with $K\equiv1$.

{\bf The Riemannian case}.

When $F$ is Riemannian, i.e. it coincides with the standard Riemannian metric $F_0$ on $S^n(1)$. All geodesics
are closed, and all prime closed geodesics (i.e. the great circles) have the same length $2\pi$ and belong to the same $\hat{G}$-orbit.
Here the action of $\hat{G}=G\times U(1)$, where $G$ is the connected isometry group, on the space of all closed geodesics
$c(t)$ with $t\in\mathbb{R}/\mathbb{Z}$ is induced by that on the free loop space, i.e.
$G$ acts on the target Finsler manifold, and $U(1)$ rotates the parameter.

The known examples of closed Finsler manifold with only one orbit of prime closed geodesics are
compact rank-one symmetric spaces. This observation inspire us to
ask if they are the only ones. A partial answer for this rigidity problem from the positive side has been given in homogenous Finsler geometry (see Theorem 1.4 in \cite{Xu2018-2}).

{\bf The Randers case}.

When $F$ is non-Riemannian Randers, it is defined by the navigation process
with the datum $(F_0,V)$, in which $F_0$ is the standard Riemannian metric on $S^n(1)$, and $V$ is a nonzero Killing vector field \cite{BRS}. The homogeneity of $F$ requires that $V$ is of constant $F_0$-length. In this case $n=2k-1>2$ is an odd number,
the connected isometry group $G=I_0(S^n,F)=U(k)$, and $V$ is defined by a vector in $\mathfrak{c}(\mathfrak{g})$ with $\mathfrak{g}=u(k)$. By \cite{PM2018,HM2011},
The affect of the Killing navigation process on the geodesics can be explicitly described.

Notice that $\pm V$ are Killing vector fields of constant length for $F$, and they generates the center $S^1$ of $I_0(S^n,F)=U(k)$.
So each integration curve of $\pm V$ is a closed geodesic on $(S^n,F)$.
We denote $l_\pm$ the lengths of the prime closed geodesics generated by $\pm V$. It is well known that $l_{+}^{-1}+l_{-}^{-1}=\pi^{-1}$. To be more self contained, we propose a proof of it which do not require $F$ to be Randers and thus can be applied to later discussion.

\begin{lemma} \label{lemma-4}
Let $(S^n,F)$ be a Finsler sphere with $K\equiv1$.
Suppose that both $c(t)$ for $t\in[0,1]$ and $c(-t)$ for $t\in[-1,0]$ are prime closed geodesics with constant $F$-speeds, which lengths are denoted as $l_+$ and $l_-$ respectively. Then we have $l_+^{-1}+l_-^{-1}=\pi^{-1}$.
\end{lemma}

\begin{proof}
Assume $c(a)$ for $a\in(0,1)$ is the image $\psi(c(0))$ for the
antipodal map $\psi$ of $(S^n,F)$. The arc length of $c(t)$ for $t\in[0,a]$ is $\pi$, while that for $t\in[0,1]$ is $l_+$. Because
$c(t)$ has a constant speed, we have $1/l_+=a/\pi$. For a similar reason, $1/l_-=(1-a)/\pi$. Adding these two equalities, then the lemma is proved.
\end{proof}

When the $F_0$-length of $V$ is an irrational multiple of $\pi$,
there are no other prime closed geodesics except those two $\hat{G}$-orbits
of prime closed geodesics generated by $\pm V$, which lengths $l_\pm$ are irrational multiples of $\pi$.
This is an important basic model for studying Finsler spheres with $K\equiv1$ and only finite orbits of prime closed geodesics \cite{Xu2018-2}.

When the $F_0$-length of $V$ is a rational multiple of $\pi$, the
antipodal map $\psi$ has a finite order $m>2$. All geodesics are closed.
The prime closed geodesics generated by $\pm V$ satisfies that their lengths $l_\pm$ are rational multiples of $\pi$, and
$l_\pm\in(\pi,m\pi]$. Because $1/l_++1/l_-=1/\pi$, the integration curves of $\pm V$ provides one or two $\hat{G}$-orbits of short prime closed geodesics. All other prime closed geodesics have the same length $m\pi$.

{\bf The non-Randers case}.

By Theorem \ref{main-thm} and Theorem \ref{main-thm-2}, the homogeneous Finsler sphere $(S^n,F)$ with $n>1$ and $K\equiv1$
may be non-Randers only when there exists a positive integer $k$, such that $n=4k-1$, $G=I_0(S^n,F)=Sp(k)$ and $S^n=Sp(k)/Sp(k-1)$.
Meanwhile, the antipodal map $\psi$ must have a finite order $m>2$.

As in the proof of Theorem \ref{main-thm}, we may identify the  homogeneous Finsler sphere $(S^n,F)$ as the unit sphere $S^{4k-1}(1)\subset\mathbb{H}^k$, on which we have the transitive
$Sp(k)$-action induced by the left $Sp(k)$-multiplications on column vectors in $\mathbb{H}^k$. Then the metric $F$ is defined on $S^{4k-1}(1)$ which is $Sp(k)$-invariant.

We will prove the following theorem in the next section, which
implies that many properties for the behavior of geodesics on
a homogeneous Randers Finsler sphere with $K\equiv1$ can be generalized to the non-Randers case. It is also a homogeneous
analog of Theorem 2 in \cite{BFIMZ}.

\begin{theorem}\label{main-thm-3}
Let $F$ be a $Sp(k)$-invariant Finsler metric on  $S^{4k-1}(1)$
such that it has constant
flag curvature $K\equiv1$ and the order of its antipodal map is a finite number $m>2$. Then we have the following:
\begin{description}
\item{\rm (1)} The antipodal map $\psi$ generates a subgroup $\mathbb{Z}_m$ in the right $Sp(1)$-multiplications on
    $S^{4k-1}(1)\subset \mathbb{H}^k$. The metric $F$ is homogeneous with respect to the action of $Sp(k)\mathbb{Z}_m\subset Sp(k)Sp(1)\subset SO(4k)$.
\item{\rm (2)} There exists a $Sp(k)$-invariant vector field $V$ on $S^{4k-1}(1)$, such that the integration curves of $\pm V$ are the only closed geodesics for both $F$ and the standard Riemannian metric $F_0$ on $S^{4k-1}(1)$.
\item{\rm (3)} Denote $l_\pm$ the lengths of the prime closed geodesics generated by $\pm V$, then $l_{\pm}\in (\pi,m\pi]$ are rational multiples of $\pi$, where $m>2$ is the order of
    the antipodal map for $F$. In particular, we have $l_+^{-1}+l_-^{-1}=\pi^{-1}$.
\item{\rm (4)} All geodesics are closed and all prime closed geodesics which are not integration curves of $\pm V$ have
    the same length $m\pi$.
\end{description}
\end{theorem}

\section{Proof of Theorem \ref{main-thm-3}}
As preparation, we first discuss
an $Sp(1)$-homogeneous Finsler sphere $(S^3(1),F)$ such that the flag curvature $K\equiv1$, and the antipodal map $\psi$ has a finite order $m>2$. The unit sphere $S^3(1)$, as well as $G=Sp(1)$, can be identified as the subset of quaternion numbers with norm one, and the $G$-action
on $S^3(1)$ is the left multiplication.
The metric $F$ is invariant under left multiplications of $Sp(1)$.

For any one-parameter subgroup $g_t\in Sp(1)$, the orbits of the left and right multiplications by $g_t$ on $S^3(1)$ are great circles. So any
$Sp(1)$-homogeneous geodesic on $(S^3(1),F)$ is a great circle.

Let $\psi$ be the antipodal map, and assume $\psi(1)=\alpha\in S^3(1)\subset\mathbb{H}$. Because $\psi$ commutes with all
the left $Sp(1)$-multiplications, we get
$\psi(q)=q\alpha$ for any
$q\in S^3(1)$.
Obviously $\alpha$ is a primitive $m$-root of $1$, i.e. $\alpha^m=1$ and $\alpha^i\neq1$ when $0<i<m$,
because the order of $\psi$ is $m$.

To summarize, we get

{\bf Claim I}. There exists a primitive $m$-th root of $1$, $\alpha\in Sp(1)$, such that $\psi(q)=q\alpha$ for any $q\in S^3(1)$.


By Claim I, $\psi$ generates a cyclic subgroup $\mathbb{Z}_m$ of right $\alpha^i$-multiplications for all $i\in\mathbb{Z}$, and the metric $F$ is homogeneous with respect to the action of $Sp(1)\mathbb{Z}_m\subset Sp(1)Sp(1)=SO(4)$.

Denote $\tilde{G}$ this subgroup $Sp(1)\mathbb{Z}_m\subset SO(4)$, its isotropy subgroup
$\tilde{H}$ at $1\in S^3(1)$ is isomorphic to $\mathbb{Z}_m$ when $m$ is odd, or $\mathbb{Z}_{m/2}$ when $m$ is even. By the assumption $m>2$, $\tilde{H}$
is always nontrivial. The isotropy action of $\tilde{H}$ splits the Lie algebra $\mathfrak{g}=sp(1)=\mathrm{Im}\mathbb{H}$ as a linear direct sum
$\mathfrak{g}=\mathfrak{m}_0+\mathfrak{m}_1$, which is orthogonal with respect to the Killing form,
so that $\tilde{H}$ acts trivially on the one dimensional subspace $\mathfrak{m}_0$ and rotates the two dimensional subspaces
$\mathfrak{m}_1=[\mathfrak{m}_0,\mathfrak{g}]$.

Let $v$ be any nonzero vector in $\mathfrak{m}_0$. Then for any $g\in\tilde{H}$, any $u\in\mathfrak{g}$,
we have $[v,u]\in\mathfrak{m}_1$ and
\begin{equation}\label{0000}
\langle v,[v,u]\rangle_v=\langle g\cdot v,g\cdot [v,u]\rangle_{g\cdot v}=\langle v,g\cdot [v,u]\rangle_v.
\end{equation}
Take the sum of (\ref{0000}) for all $g\in \tilde{H}$, and apply the equality $\sum_{g\in\tilde{H}}g=0$,
we get $\langle v,[v,u]\rangle_v=0$ which implies that the one-parameter subgroup $\exp tv$ generated by $v$ is
an $Sp(1)$-homogeneous geodesic on $(S^3(1),F)$.

To summarize, we get

{\bf Claim II}. Any nonzero vector $v\in\mathfrak{m}_0$ generates
a one parameter subgroup $\exp tv\subset Sp(1)$ which is a geodesic on $(S^3(1),F)$ for both directions.

Denote $V$ the vector field defined by
$V(q)=qv$, $\forall q\in S^3(1)$, for any nonzero $v\in\mathfrak{m}_0$.
Because $V$ is
left invariant and $\alpha\in e^{\mathbb{R}v}$, it is easy to see that $\psi$ is contained in the flow generated by $V$, and the integration curves of $V$ provide two orbits of homogeneous geodesics.

As point sets, there exists no other homogeneous geodesics on $(S^3(1),F)$. The reason is the following.
Any homogeneous geodesics on $(S^3(1),F)$ is a great circle.
If it contains some $q\in Sp(1)$, it contains the $\psi$-orbit of $q$, which has more than two points. So the great circle passing them is unique.

Now we are ready to discuss the general case and prove Theorem \ref{main-thm-3}.

\noindent{\bf Proof of Theorem \ref{main-thm-3}}.

(1) Denote $x_0=(0,\ldots,0,1)^T\in S^{4k-1}(1)\in \mathbb{H}^k$, the isotropy subgroup $H=Sp(k-1)$ corresponds to the left-up $(k-1)\times(k-1)$-block in
$G=Sp(k)$. As $\psi$ commutes with $G$, i.e. the left $Sp(k)$-multiplications, each point $\psi^k(x_0)$ in the $\psi$-orbit of $x_0$ is fixed by $H$. So $\psi(x_0)$ is contained
in the fixed point set
$$\mathrm{Fix}(H,S^{4k-1}(1))=S^3(1)
=\{(0,\ldots,0,q)^T|\,\forall q\in\mathbb{H}\mbox{ with }|q|=1\}.$$
We may assume $\psi(x_0)=(0,\ldots,0,\alpha)^T$, then the $Sp(k)$-invariance of $\psi$ implies $\psi(x)=x\alpha$ for any column vector $x\in
S^{4k-1}(1)\in\mathbb{H}^k$. Meanwhile, we see $\alpha\in Sp(1)\in\mathbb{H}$ is a primitive $m$-th root of $1$.
So $\psi$ generates a subgroup $\mathbb{Z}_m$ of the
right scalar multiplications, and the metric $F$ is
homogeneous with respect to the action of $Sp(k)\mathbb{Z}_m
\subset Sp(k)Sp(1)\subset SO(4k)$.

This proves (1) of the theorem.

(2) The fixed point set $S^3(1)=\mathrm{Fix}(H,S^{4k-1})$
is totally geodesic $Sp(1)$-homogeneous submanifold in $(S^{4k-1}(1),F)$, so it has constant flag curvature $K\equiv1$,
and its antipodal map coincides with the restriction of $\psi$,
which has the same finite order $m>2$. Claim II provides a vector field $V((0,\ldots,0,q)^T)=(0,\ldots,0,qv)^T$ for some $v\in sp(1)$, which integration curves are homogeneous geodesics in both directions. The vector field $V$ can be naturally extended by left $Sp(k)$-invariance to $S^{4k-1}(1)$, such that
$V(x)=xv$ for any $x\in S^{4k-1}(1)\in\mathbb{H}^k$, $\psi$ is contained in the flow generated by $V$, and
all integration curves of $\pm V$ are geodesics for both $F$ and the
standard Riemannian metric $F_0$.

Because the great circle containing any $\psi$-orbit is unique,
there does not exist any other geodesics on $S^{4k-1}(1)$ for both $F$ and  $F_0$.

This ends the proof of (2).

(3) Let $l_\pm$ be the lengths of the prime closed geodesics
generated by $\pm V$, where $V$ is the nonzero left invariant
vector field in (2). Because the order of $\psi$ is $m>2$,
and $\psi$ is a Clifford--Wolf translation, any geodesic segment of length $m\pi$ on $(S^{4k-1}(1),F)$ is a closed geodesic.
So we have $l_\pm\leq m\pi$. On the other hand, $l_\pm>\pi$ is obvious.
By Lemma \ref{lemma-4}, we get the equality
$l_+^{-}+l_-^{-1}=\pi^{-1}$.

This proves (3) of Theorem \ref{main-thm-3}.

(4) Assume that $c(t)$ with $t\in\mathbb{R}$ is an $F$-unit speed geodesic, which provides a prime closed geodesic of length $l\in(\pi,m\pi)$. Then we only need to prove $c(t)$ is an integration curve of the vector field $V$ (or $-V$) in (2).

Obviously we have
$\psi^i(c(t))=c(t+i\pi)$ and $c(t)=c(t+jl)$ for any $i,j\in\mathbb{Z}$ and $t\in\mathbb{R}$.
Because $c(t)$ for $t\in[0,m\pi]$ is also a closed geodesic, we have $m\pi=m'l$ for some positive integer $m'$. The integers $m$ and $m'$ must be co-prime to each other, otherwise by the homogeneity of $F$, the order of $\psi$ will be smaller than $m$.

Denote
$t_i=i\pi-\lfloor i\pi/l\rfloor l$ for $i\in\{1,\ldots,m\}$, in which $\lfloor a\rfloor$ for $a\in\mathbb{R}$
is the largest integer smaller than $a$. Because the $\psi$-orbit $\psi^i(c(0))=c(t_i)$ for $1\leq i\leq m$ contains exactly $m$ points, we have
$\{t_1,\ldots,t_m\}=\{l/m,2l/m,\ldots,l\}$. Because $l<m\pi$ and $t_m=l$, we can find some integer $j$, $1<j<m$, such that $t_j=l/m<\pi$. The geodesic $c(t)$ for $t\in[0,l/m]$ is the shortest geodesic from $c(0)$ to $\psi^j(c(0))$.

Now we switch to the $Sp(k)$-invariant vector field $V$
in (2). Then the integration curve of $V$ passing $c(0)$ is
a reversible geodesic containing the $\psi$-orbit of $c(0)$, with positive constant $F$-speed for both directions. We can suitably choose $V$ or $-V$, such that its integration curve from $c(0)$ to $\psi^j(c(0))$ does not pass $\psi(c(0))$. Then it is the shortest geodesic from $c(0)$ to
$\psi^j(c(0))$ and $d_F(c(0),\psi^j(c(0)))<\pi$.

Because $d_F(c(0),\psi^j(c(0)))<\pi$, the shortest geodesic from $c(0)$ to $\psi^j(c(0))$ is unique. So as point sets,
the geodesic $c(t)$ coincides with an integration curve of $\pm V$.

This proves (4) in the theorem.
\ \rule{0.5em}{0.5em}

{\bf Acknowledgement.} The author sincerely thank Vladimir S. Matveev, Yuri G. Nikonorov and Wolfgang Ziller for helpful discussions and suggestions. This paper is supported by National Natural Science Foundation of China (No. 11821101, No. 11771331), Beijing Natural Science Foundation
(No. 00719210010001, No. 1182006), Capacity Building for Sci-Tech  Innovation -- Fundamental Scientific Research Funds (No. KM201910028021).

\end{document}